\documentclass[a4paper,11pt]{amsart}
\usepackage{graphicx}
\usepackage{amsmath}
\usepackage{amssymb}
\usepackage{oldgerm}
\usepackage[all]{xy}

\newcommand{\Hom}{\operatorname{Hom}\nolimits}

\renewcommand{\mod}{\operatorname{mod}\nolimits}

\newcommand{\id}{\operatorname{id}\nolimits}
\newcommand{\Ext}{\operatorname{Ext}\nolimits}
\newcommand{\rad}{\operatorname{rad}\nolimits}
\newcommand{\Z}{\operatorname{\mathbb{Z}}\nolimits}
\newcommand{\N}{\operatorname{\mathbb{N}}\nolimits}
\newcommand{\aut}{\operatorname{Aut}\nolimits}
\newcommand{\soc}{\operatorname{soc}\nolimits}
\newcommand{\Irr}{\operatorname{Irr}\nolimits}
\newcommand{\Sp}{\operatorname{Span}\nolimits}
\newtheorem{theo}{Theorem}[section]

\newtheorem{cor}[theo]{Corollary}

\newtheorem{lemma}[theo]{Lemma}
\newtheorem{prop}[theo]{Proposition}
\newtheorem{defi}[theo]{Definition}

\newtheorem{example}[theo]{Example}

\newtheorem{as}[theo]{Assumption}

\begin{document}
\baselineskip=15pt
\title[Euclidean components for self-injective algebras]{Euclidean components for a class of self-injective algebras}
\author{Sarah Scherotzke (Oxford)}
\address{Sarah Scherotzke \newline Mathematical Institute \\ University of Oxford \\ 24-29 St.\
Giles \\ Oxford OX1 3LB \\ United Kingdom}

\email{scherotz@maths.ox.ac.uk}

\keywords{Auslander-Reiten theory, $G$-transitive algebras }
\subjclass[2000]{16G70}

 \maketitle

\begin{abstract}We determine the length of composition series of projective modules
of $G$-transitive algebras with an Auslander-Reiten component of
Euclidean tree class. We thereby correct and generalize a result in \cite[4.6]{F2}.
 Furthermore we show that modules with certain length of composition series are periodic.
We apply these results to $G$-transitive blocks of the universal
enveloping algebras of restricted $p$-Lie algebras and prove that
$G$-transitive principal blocks only allow components with
Euclidean tree class if $p=2$. Finally we deduce conditions for a
smash product of a local basic algebra $\Gamma$ with a commutative
semi-simple group algebra to have components with Euclidean tree
class, depending on the components of the Auslander-Reiten quiver
of $\Gamma$.
 \end{abstract}

\section{introduction}
The stable Auslander-Reiten quiver of a finite-dimensional algebra can be viewed as part
of a presentation of the stable module category, and it is an important invariant which has many applications.
It is a locally finite graph, where the vertices correspond to the isomorphism classes of indecomposable modules.
Each connected component is isomorphic to $\Z[T]/\Gamma$ where $T$ is a tree, and $\Gamma$ is an
admissible group of automorphism. (See \cite{B}[4.15.6] for details).

For many self-injective algebras it is known that the
possibilities for $T$ are restricted, it can only be  Dynkin, or
Euclidean, or one of a few infinite trees (see \cite{W}, \cite{E},
\cite{ES}). In this paper we study algebras with Euclidean
components. Recently the study of self-injective algebras with
Euclidean Auslander-Reiten components has attracted much
attention, for example all self-injective algebras of Euclidean
type have this property (see the survey article \cite[Section
4]{Sk} ).

These have been also studied in the context of reduced enveloping
algebras by Farnsteiner, and we discovered that \cite[4.6]{F2} is
not correct. In Theorem \cite[4.6]{F2} Farnsteiner proves a
necessary conditions for certain blocks of universal enveloping
algebras $u(L,\chi)$ of a restricted $p$-Lie algebra $L$ with
$p>2$, to have an Auslander-Reiten component with Euclidean tree
class. Unfortunately one crucial step in the proof  is wrong. As
all other results in the case of $\tilde D_n$-tree class depend on
this step, we need a different proof and Theorem.

\medskip

In the first section we give
a proof for the more general setup of $G$-transitive blocks of
Frobenius algebras in any characteristic.

\smallskip

In the second section we apply the main result of the first
section to $G$-transitive blocks of $u(L,\chi)$. We show that some
of the results of Farnsteiner's paper remain true while others
need additional assumptions. We can show that a
$G$-transitive principal block of $u(L, \chi)$ does not have
Auslander-Reiten components of Euclidean tree class if $p>2$.

\smallskip

In the last section we determine conditions for the smash product
of a basic local algebra $\Gamma$ and a semi-simple commutative
group algebra to have Auslander-Reiten components of Euclidean
tree class depending on the tree class of components of $\Gamma$.

\medskip

Let $B$ be an indecomposable Frobenius algebra. We introduce for a group $G\subset
\aut(B)$ the $G$-transitive algebra $B$. This means that $G$ acts transitively on the set
 of simple modules in $B$. We denote by $l(P)$ the length of an indecomposable projective module $P$.
 We show in \ref{euc1}, \ref{euc2}, \ref{5} that the following holds for $G$-transitive algebras
that have Auslander-Reiten components of Euclidean tree class:

(1) All non-periodic Auslander-Reiten components are either
isomorphic to $\Z [\tilde A_{1,2}]$ or $\Z[ \tilde D_n]$ for $n$ odd and $n>5$.

(2) In the first case $l(P)=4$ and all indecomposable modules of length $0 \mod l(P)$ and $2 \mod l(P)$ are periodic.

(3) In the second case $l(P)=8$ and all indecomposable modules of length $4 \mod l(P)$ are periodic.

\smallskip

In \cite{F2} Farnsteiner introduces $G(L)$, the group of
group-like elements of $u^*(L)$ and shows that they can be
embedded into the automorphism group of $u(L,\chi)$. He proves
 for $G(L)$-transitive blocks over a
field of characteristic $p>2$ (see \cite[4.6]{F2}) that:

(1) All non-periodic Auslander-Reiten components are either
isomorphic to $\Z [\tilde A_{1,2}]$ or $\Z[ \tilde D_5]$. All indecomposable modules of length 2 are periodic.

(2) In the first case $l(P)=4$.

(3) In the second case $l(P)=8$.

For his results in the case of tree class $\tilde D_n$ he first
shows that $n=5$. As this step is wrong, a different proof was
needed for the more general setup. In the new proof we can show
that $n>5$ which contradicts Farnsteiner's result.

\smallskip

Additionally we show that that the number of non-periodic
components is equal to the number of isomorphism types of simple
modules in $B$.

\medskip

In the second section we can verify Farnsteiner's statement that
supersolvable algebras in characteristic $p>2$ do not have
Auslander-Reiten components of Euclidean tree class.
 An example by \cite[2.3]{E} shows that this statement is wrong for $p=2$.
 We also show in \ref{lie} that if $B$ is a $G$-transitive block of $u(L,\chi)$ that has
an Auslander-Reiten component of Euclidean tree class, then $p=2$
or the dimension of a simple module in $B$ is divisible by 2.

\medskip

In the last section we introduce the smash product of a basic local algebra $\Gamma$
and a semi-simple group algebra $kG$, where $G\subset \aut(\Gamma)$ is an abelian group.
 We show that they are a special cases of $G$-transitive algebras and describe how to construct the Gabriel
 quiver of the smash product
from the Gabriel quiver of $\Gamma$.

For a finite-dimensional algebra $A$, we introduce $G(A)$ the free
abelian group with basis the isomorphism types of indecomposable
modules. We have a bilinear form $(-,-)_A$ induced by $\dim_k
\Hom_A(-,-)$ defined on $G(A)$. We develop some properties of this
bilinear form in \ref{dual}, \ref{sum}.

We use these general results to show in \ref{smash euc} that Auslander-Reiten sequences of
the smash product restricted to $\Gamma$ are sums of Auslander-Reiten sequences
in $\Gamma$ that are twists of each other. Using this result, we
prove that the smash product has an Auslander-Reiten component of
tree class $\tilde D_n$ if and only if $\Gamma$ has an
Auslander-Reiten component of tree class $\tilde D_n$. If the
smash product has an Auslander-Reiten component of tree class
$\tilde A_{1,2}$, then so does $\Gamma$. The converse is not true,
as we show by providing a counter example. We can in this case restrict the component to $Z[\tilde A_{12}]$ and $\Z[\tilde A_{n}]$ for certain $n \in \N$ (see \ref{12}).

\bigskip

I would like to thank my Supervisor Karin Erdmann for her comments which have improved this a paper considerably. I would also like to thank
the referee and the editor Andrzej Skowro\'nski for their useful suggestions.

\section{Euclidean components for $G$-transitive Algebras}
For general background on Auslander-Reiten theory we refer to
\cite{A} or \cite{B}. Let $F$ be a field, $B$ be an indecomposable
Frobenius algebra over $F$ with Nakayama automorphism $\nu$ and
let $P$ be a projective indecomposable $B$-module. We denote by
$\tau$ the Auslander-Reiten translation of $B$. As $B$ is a
Frobenius algebra, we have $\tau\cong \Omega^2 \nu$ by
\cite[4.12.9]{B}. Furthermore we denote by $B$-$\mod$ the category
of finite-dimensional left $B$-modules.

Let $ \alpha \in \aut(B)$ and $M, \ N \in B$-$\mod$. Then we denote by $M_{\alpha}$ the module which is isomorphic to $M$
 as an abelian group and the action of $B$ on $M_{\alpha}$ is given by $b.m:=\alpha (b)m$ for all $b\in B$ and $m
 \in M$. We denote by $l(M)$ the length of composition series of $M$, by $c(M)$ its complexity and by $\Irr(M,N)$ the space of irreducible maps from $M $ to
$N$.

\medskip

Let $G$ be a subgroup of $\aut(B)$. We call $B$ a $G$-transitive
block, if for any two simple left $B$-modules $V$ and $W$ there is
an element $g \in G$ such that $V_g\cong W$. We denote by $T_s(B)$
the stable Auslander-Reiten quiver of $B$.

\smallskip

>From now on we assume that $B$ is $G$-transitive.

\smallskip

We first prove a result on the length of the modules appearing at the end of an Auslander-Reiten sequence.
\begin{lemma}\label{length} Let $M$ be an indecomposable, non-projective $B$-module. Then $l(\tau( M))=l (M)+n l(P) $ for some $n \in \Z$.
\end{lemma}
\begin{proof}
As $B$ is $G$-transitive all projective indecomposable modules of $B$ have the same length.
 Therefore the length of projective covers of left $B$-modules are multiples of $l(P)$.
 It follows that for all $i\in \N$ there is a $n\in \Z$ such that $l(\Omega^i(M))=l(M)+nl(P)$.
  Therefore $l (\tau(M))=l(\Omega^2(M_{\nu^{-1}}))= l(M_{\nu^{-1}})+nl(P)=l(M)+nl(P)$ for some $n\in \Z$.
\end{proof}

The next lemma proves a condition which ensures that all indecomposable length 2 modules of a block have complexity one.

\begin{lemma}\label{period} Suppose $P$ has length 4. Then
every indecomposable $B$-module of length 2 has complexity one.
\end{lemma}
\begin{proof}Let $M$ be an indecomposable $B$-module of length 2. The module $M$ has a simple top and therefore an indecomposable projective cover.
 Then $\Omega(M)$ is also an indecomposable length 2 module and has therefore an indecomposable projective cover. If follows that $l (\Omega^n(M))=2$ for all $n \in \N$. Therefore the complexity of $M$ is one.

\end{proof} In general not every module of a Frobenius algebra with complexity 1 is periodic. In particular it does not need to hold for an algebra with Auslander-Reiten component
 of Euclidean tree class, as is shown in the next
\begin{example}\cite{S}
Let $A_q:=F\langle x,y \rangle /(x^2,y^2, xy-qyx)$ where $q\not =0$ and where
$q$ is not a root of
unity. Let $M_{\gamma}=\Sp \{v, xv\}$ be the two-dimensional module with
$yv=\gamma xv$ for $\gamma \in F^*$. The projective cover of
$M_{\gamma}$ is given by $\pi:A_q \to M_{\gamma}$ with $\pi(1)=v$. Then
$\Omega(M_{\gamma})=\Sp\{ xy, y-\gamma x \}\cong  M_{\gamma q}$.  As $q$ is not
a root of unity we have $\Omega^k(M_{\gamma})\not \cong M_{\gamma}$ for
all $k \in \N $. Therefore $M_{\gamma}$ is not periodic but has complexity one.
Furthermore the Auslander-Reiten component containing the simple module is isomorphic to $\Z[\tilde A_{12}]$.
\end{example}

In the case of Auslander-Reiten components with Euclidean tree
class, we know that the simple modules are non-periodic.

\begin{lemma}\label{simple}
Suppose that $T_s(B)$ has a
component $\theta$ of Euclidean tree class. Then all simple
modules are non-periodic and lie in an Auslander-Reiten component
isomorphic to $\theta$.
\end{lemma}
\begin{proof} As $\theta$ has Euclidean tree class it is attached to a
projective indecomposable module $P$ by \cite[2.4]{W} and
\cite[IV, 5.5]{A}.  If $l(P)< 4$, then $l(P)$ is uniserial which is a
contradiction to the Euclidean tree class by \cite[4.16.2]{B}. As
$P$ is attached to $\theta$ the indecomposable module $P/\soc P$
lies in $\theta$. The map induced by $\Omega$ restricted to
$\theta$ is an isomorphism. Therefore $\Omega(P/\soc P)=\soc P$ is
contained in a component isomorphic to $\theta$. So $\soc (P) $ is
not periodic. As $B$ is $G$-transitive, all simple modules are
non-periodic and lie in components isomorphic to $\theta$.
\end{proof}

We define certain stable graph automorphisms for $G$-transitive
blocks with an Auslander-Reiten component $\theta$ of Euclidean tree class. Note that by \cite[2.4]{W} and
\cite[IV, 5.5]{A}, there is at least one projective indecomposable module
$P$ attached to $\theta$.
Those maps have been defined in the proof of \cite[4.6]{F2}
similarly.

\begin{defi}\label{phi}

Suppose that the stable Auslander-Reiten quiver $T_s(B)$ has a component
$\theta$ of Euclidean tree class. Define $\phi_g:T_s(B) \to T_s(B)$ by $ M
\mapsto  \Omega(M_{g})$. Then $ \phi_g$ is a stable graph
isomorphism. Define  for any $g \in G$ the map $A_g: T_s(B)
\to  T_s(B)$, by $M \mapsto M_g$. We denote by $\theta_g$ the
component of $T_s(B)$ which is the image of $A_g(\theta)$.

For the rest of the section we fix for any component $\theta$ with
Euclidean tree class a projective indecomposable module $P$ that
is attached to $\theta$ and an element $g \in G$ such that $\phi:=
\phi_g|_{\theta} $ is an automorphism of $\theta$. Furthermore let
$S:=\phi(P/\soc P)$.
\end{defi}
Also $\phi_g|_{\theta}$ is an automorphism of
$\theta$ if and only if $ P_{g^{-1}} $ is attached to
$\Omega(\theta) $. We can see this as follows. If $P_{g^{-1}}$ is attached to $\Omega(\theta)$, then $A_g$ induces an isomorphism from
$\Omega(\theta) $ to $\theta$ and $\phi_g|_{\theta}$ is an automorphism.

Suppose $\phi_g|_{\theta}$ is an automorphism. Then $A_{g^{-1}}$ induces an isomorphism from $\theta$ to $\Omega(\theta)$.
 As $P$ is attached to $\theta$, $P_{g^{-1}}$ is attached to $\Omega(\theta)$.

Note that $S$ is a simple module that belongs to
$\theta$.

\smallskip

First we need to show that the twisting action of $\nu$ commutes
with the twisting action of any automorphism of $B$.
\begin{lemma}Let $A$ be a Frobenius algebra. For all $g \in \aut (A)$ and $A$-modules $M$ we have $M_{\nu g}\cong M_{g \nu}$.
\end{lemma}
\begin{proof}Let $(-,-)$ and $\{-,-\}$ be two associative non-degenerate bilinear forms, $\nu$ and $\nu_1$ the corresponding Nakayama automorphisms and
let $f:=(-,1)$ and $f_1:=\{-,1\}$ be the corresponding linear
forms. Then $\pi:A \to A^*, a \mapsto af$, and $\pi_1: A \to A^*,
a \mapsto af_1$, are $B$-module isomorphisms.
 Therefore there exist $x,y \in A$ such that $xf=f_1$ and $yf_1=f$.
  It follows that $x=y^{-1}$.  Set $u:=\nu(x).$ Then the following equation holds:
   $\{a,b\}=(ab,x)=(a,bx)=(\nu(b)u,a)=\{\nu(b)u,ax^{-1}\}= \{u^{-1}\nu(b)u,a\}$
 for all $a,b \in A$. Let $C_u: A\to A$, $a\to u^{-1} a u$ for all $a\in A$. Then  $\nu_1= C_u \circ \nu $.

\smallskip

Let $g \in \aut(A)$, then $\{-,-\}:=(-,-)\circ(g\times g)$ is a
associative non-degenerate bilinear form. It has Nakayama
automorphism $g^{-1}\circ \nu\circ g$ as the following holds:
$\{a,b\}=(g(a),g(b))=(\nu(g(b)), g(a))=\{g^{-1}(\nu(g(b))),a\}$
for all $a,b \in A$.  By the first part there exists an invertible
element $u \in A$ such that $g^{-1}\circ \nu\circ g=C_u \circ \nu
$. Therefore $M_{g\nu g^{-1}}\cong M_{C_u \circ \nu}$ for all left
$B$-modules $M$. But $M_{\nu} \cong M_{C_u \circ \nu}$ via the
automorphism $\phi: M_{\nu} \to M_{C_u \circ \nu }, m\mapsto
u^{-1}m$ and therefore we conclude $ M_{g \nu} \cong M_{\nu g}$.
\end{proof}

\section{Restrictions on tree classes}
We say for the rest of the article that the algebra $A$ satisfies
(C) if

\medskip

{\it Every module with complexity $1$ is $\Omega$-periodic }

\medskip

We say an algebra satisfies $(C')$ if

 \medskip

{\it Every module with complexity $1$ is $\Omega$ and $\tau$-periodic }

\medskip
Note that if $\nu$ has finite order and (C) holds then (C') is also true.
 We introduce the following
\begin{as}\label{as}
For $B$, we assume that all elements in $G$ have finite order.
Furthermore we assume that $B$ satisfies (C') and that $T_s(B)$ has a component $\theta$ of
Euclidean tree class. Let $P$ be a projective indecomposable module attached to $\theta$.
\end{as}

We have the following condition for the existence of a non-periodic
indecomposable module of length 3.

\begin{lemma} \label{uni} Suppose $B$ satisfies $(C)$. If all indecomposable modules of length 2 have complexity one, then all indecomposable modules of length 3 are non-periodic. Also there is no uniserial module of length 3. \end{lemma}
\begin{proof}Every indecomposable module of length 3 has a simple
top or a simple socle. Let $M$ be an indecomposable module of
length 3 with simple top. Such an element exists, because a factor
module of $P$ of length 3 has simple top and is therefore
indecomposable. Then there exists an exact sequence $0\to S \to M
\to L\to 0$ such that $L$ is an indecomposable module of length 2
and $S$ is a simple module. Then $c(S)\le \max \{ c(M), c(L) \} $.
By \ref{simple} $S$ is non-periodic and therefore $c(S)\ge 2$. As
$c(L)=1$, we have $c(M)\ge 2$ and therefore $M$ has to be non-periodic. If
$M$ has a simple socle we can find an indecomposable module $N$ of
length 2 such that $0\to N \to L \to S \to 0$ is an exact
sequence. By the same argument as in the previous case, $M$ has to
be non-periodic.

We assume that there is a uniserial module $M$ of length 3 with composition series
$S_1, S_2 , S_3$. Let $L= \rad M$ and $N= M/ S_3$. Then an exact sequence is given by
$0\to L \to S_2\oplus M \to N \to 0$. As $L$ and $N$ are indecomposable modules of length
2 they are periodic. Therefore $c(M)\le 1$ and $M$ is therefore periodic. This is a contradiction to the first part.
\end{proof}
The proof of the next Theorem goes along the lines of the proof of
\cite[4.6]{F2}. As the setup here is more general and the author
uses properties of the universal enveloping algebra of restricted
$p$-Lie algebras, we give a proof for our setup.

\begin{theo}\label{euc1} Let $B$ be as in \ref{as}. Then the following statements hold:

(1)  $\theta $ is isomorphic to $\Z [\tilde A_{12}]$ or $\Z
[\tilde D_n]$ with $n$ odd.

(2) The group $G$ acts transitively on the non-periodic components.

(3) If $\theta \cong \Z [\tilde A_{12}]$, then all projective indecomposable
 left $B$-modules have length 4.
 Furthermore all indecomposable modules of length $2\mod 4$ and all
 indecomposable non-projective modules of length
$0 \mod 4$ are periodic.
\end{theo}
\begin{proof}Let $T$ be the tree class of $\theta$. Then $\phi$
induces a graph automorphism $f:T \to T$. Suppose $f$ has a fixed
point. Then there is an indecomposable module $M$ in $\theta$ such
that $\Omega^{2m}(M_{\nu^{-m}})\cong \tau^m(M) \cong \Omega(M_g)$.
Therefore $M$ has complexity 1 and is by assumption $\tau$-periodic which
is a contradiction. Therefore $f$ does not have fixed points. The
only Euclidean trees which admit an automorphism without fixed
points are $\tilde A_{12}$ and $(\tilde D_n)_{n\ge 4}$ with $n$
odd.  We have  $\theta \cong \Z [\tilde A_{1,2}]$ or $\theta \cong
\Z [\tilde D_n]$ by \cite [2.1]{F2}. Furthermore all
indecomposable modules which do not lie in $\Psi:= \bigcup_{w\in
G}\theta_w$ are periodic.  This can be seen as follows:  let $Y$
be an indecomposable module which is not in $\Psi$. We have $
\Omega(\Psi)=\Psi$. Therefore the function $d_{Y}:\Psi \to \N, M
\mapsto \dim_F \Ext^1(Y, M)$ is additive by \cite[3.2]{ES} and
bounded by \cite[2.4]{W}. Since all simple modules are contained
in $\Psi$  by \ref{simple} there exists an $m \in \N $ such that
$\dim_F \Ext ^n(Y,W) \le m$ for all $n \ge 1$ and all simple
modules $W$, so that $Y$ has complexity one.  Therefore $Y$ is
periodic.

To prove $(3)$ we suppose that $T=\tilde A_{12}$. Then the proof
is the same as in \cite[4.6]{F2}. For the convenience of the
reader we include a different proof. As $S$ and $P/\soc P$ are in
$\theta\cong \Z [\tilde A_{1,2}]$, all modules in $\theta $ have
length $-1 \mod l( P)$ or $1\mod l(P)$ and two modules connected
by an arrow have a different length $\mod l(P)$. The length
function $l \mod l(P)$ is additive on Auslander-Reiten sequences.
Therefore $-2 \mod l(P)=2$ and it follows that $l(P)=4.$ As no
modules of length $2 \mod l(P)$ and $0 \mod  l(P)$ occur in
$\theta$, we have by (2) that all indecomposable modules of length
$2 \mod l(P)$ and all indecomposable non-projective modules of
length $0 \mod l(P)$ are periodic.
\end{proof}
Let $\Z[\tilde D_n]$ be indexed as follows: $(k, 1), \ldots, (k,
n+1)$ denote the $k$-th copy of $\tilde D_n$ for any $k \in \Z$:
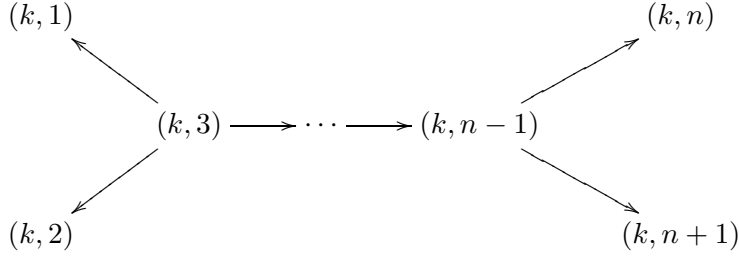
\[\xymatrix{ (k,1)& & & &(k,n)\\
& (k,3)\ar[r]\ar[ld]\ar[lu]& \cdots \ar[r]& (k,n-1)\ar[ru]\ar[rd]& \\
(k,2)& & & &(k,n+1)}\]
With this notation $\tau((k,i))=(k-1,i)$.

\smallskip

For an indecomposable non-projective module
 we denote by $\bar \alpha(M)$ the number of predecessors of $M$ in $T_s(B)$.

In order to prove our second Theorem, we need the following
\begin{lemma}\label{indec}
Assume $T_s(B)$ has a component $\theta$ which is isomorphic to $\Z [\tilde{D}_n]$ with
$n$ odd. Then

(a) \ ${\rm rad} P/{\rm soc} P$ is indecomposable.

(b) \ $l(P)$ is even.

(c) \ All $M$ with $\bar{\alpha}(M) = 2$ or $3$ have even length.

\end{lemma}
\begin{proof}

(1) We assume that $l(P)$ is even and show that (c) holds in this case. Let $M$ be a module with $\bar \alpha (M)=3$. Then there exist a module
$N$ such that $M$ is its only non-projective predecessor. Let $l(N)=a$, then $l(M)=2a
\mod l(P)$ by Lemma \ref{length}. By assumption $l(P)$ is even
and therefore all modules with 3 predecessors have even length. Consider
the following extract from $\Z[\tilde D_n]$ where $M_3:=(k,3)$ and
$M_{n-1}:=(k,n-1)$ denote the isomorphism type of modules with 3
predecessors and $M_t:=(k,t)$ for $3\le t \le n-1$:
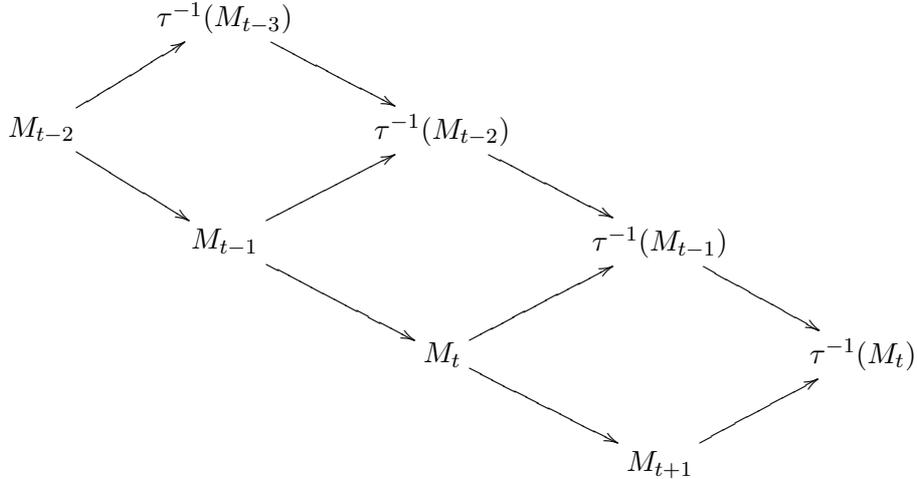
\[\xymatrix{ &  \tau^{-1}(M_{t-3}) \ar[rd]& & & \\
M_{t-2}\ar[ur]\ar[dr] & & \tau^{-1}(M_{t-2}) \ar[rd]& & \\
&  M_{t-1}\ar[ru]\ar[dr]& &\tau^{-1}(M_{t-1}) \ar[rd]\\
& & M_t\ar[ru]\ar[rd]& &\tau^{-1}(M_{t})\\
& & & M_{t+1}\ar[ru] & }\]

We assume $(c)$ is false and let $t$ be minimal such that $M_t$ has odd length. Suppose $t>4$,
 then $M_{t-2}$ is of even length and  $\bar \alpha(M_{t-1})=2$. Then $2 l(M_{t-1})= l(M_t) +l(M_{t-2}) \mod l(P)$. This gives a contradiction as the right hand side is an even number
 and the left hand side is and odd number. Therefore $ t=4$.
 Then $M_5$ is of even length and $M_6$ of odd length.  We can show that $M_{s}$ with $s$ odd has even length and $M_s$ with $s$ even has odd length. As $n$ is
odd, the module $M_{n-1}$ has to be of odd length, which is a contradiction.

(2) The Auslander-Reiten sequence ending in $P/\soc P $ is given by \[0\to
\rad P\to P\oplus \rad P/\soc P\to P/\soc P\to 0.\]
We assume that $\rad P/\soc P$ is decomposable.

Then $\bar \alpha(P/\soc P)>1$. There exists no
module $N$ such that $0\to \tau(N)\to S \to N\to 0$ is an
Auslander-Reiten sequence.  Any projective indecomposable module
$Q$ appears only as a summand of the middle term of an
Auslander-Reiten sequence with middle term $Q \oplus \rad Q/\soc
Q$ and $S \not = \rad Q/\soc Q$ because $l(Q)>3$. So there exists
no  Auslander-Reiten sequence of the form $0\to \tau(N)\to S\oplus
Q \to N\to 0$ with $Q$ non-zero. Therefore $ \bar \alpha(S) \not =3$. As $\phi$ maps $P/\soc P$ to $S$ we have $\bar \alpha(S)= \bar \alpha(P/\soc P)$. Therefore $\bar \alpha (S)=\bar \alpha (P/\soc P)=2$. For $n=5$ all modules have either one or
three predecessors, so that $n>5$. Suppose $l(P)$ is even. Then
$(1)$ shows that all modules with 2 or 3 predecessors in $\theta$ are of even
length. As $S$ is not of even length, this is a contradiction.
Therefore $l(P)$ is odd.  Let $M_1$ be a module of length $a$ with
only one predecessor $M_3$, then $l(M_3)=2a \mod l(P)$. For the
other module $M_2$ with only predecessor $M_3$ and length $\bar a$
we have $2\bar a =2a \mod l(P)$. As $l(P)$ is odd this gives us
$\bar a= a \mod l(P)$. We can deduce that $l(M_4)+2a= 4a \mod
l(P)$ and therefore $l(M_4)= 2a \mod l(P)$. It follows inductively
that $l(M_i)=2a \mod l(P)$ for all modules with 2 predecessors.
Therefore we have $-1=2a \mod l(P)$ and $ 1=2a \mod l(P)$ as
$P/\soc P$ and $S$ are modules with 2 predecessors. This is a
contradiction as $l(P)$ is odd. This proves (a).

(3) The only predecessor of $P/\soc P$ is $\rad P/\soc P$. As
$S=\phi(P/\soc P)$, we have $\bar \alpha(S)=1$. The
predecessor of $S$ has length $2 \mod l(P)$ and the predecessor of
$P/\soc P$ has length $-2 \mod l(P)$ by the argument of (2).

Suppose that $l(P)$ is odd. If $n=5$ this is a contradiction. If $n>5$, then by (2) all modules with 2 predecessors have length $2 \mod l(P)$ and $-2 \mod l(P)$ which is a contradiction as $l(P) $ is odd. Therefore (b) holds. Then (1) proves (c).
\end{proof}

We also need the following
\begin{lemma}\label{d}
Suppose $T_s(B)$ has a
component $\theta$ of tree class $\tilde D_n$. Then

(1) $l(P)> 4$.

(2) $P/ \soc P$ and $S$ have one predecessor and the $\tau$-orbit of their predecessors are different.

\end{lemma}
\begin{proof}We set $l:=l(P)$. We know by Lemma \ref{indec}
that $\rad P/\soc P$ is indecomposable.

\smallskip

It was shown in the proof of \ref{simple} that $l\ge 4$. Suppose
now that $l=4$, then by Lemma \ref{period} all indecomposable
modules of length 2 are periodic. By \ref{indec} $\rad P/\soc P$
is indecomposable and therefore periodic. As $P$ is attached to
$\theta$, the module $\rad P/\soc P$ also belongs to $\theta$
which is a contradiction by \cite[4.16.2]{B}. Therefore $l>4$.

\medskip

From $(3)$ of the proof of \ref{indec} we know that $S$ and
$P/\soc P$ have only one predecessor of length $2 \mod l$ and $-2
\mod l$ respectively. As $l \not =4$ and by \ref{length} their
predecessors do not lie in the same $\tau $-orbit.

\end{proof}
We can now deduce the length of projective indecomposable modules if the Auslander-Reiten quiver has components $\Z[\tilde D_5]$.
\begin{prop}\label{euc2,5}
Let $B$ be as in $\ref{as}.$  Suppose $T_s(B)$ has a
component $\theta$ of tree class $\tilde D_5$. Then $l(P)=8$ and all indecomposable modules of length 2 and of length 4 $\mod l(P)$ are periodic.
\end{prop}
\begin{proof} We set $l:=l(P)$.
Let $x$ be the length modulo $l$ of the module which has as only predecessor the module of length $-2 \mod  l$ and let
$y$ the length modulo $l$ of the module which has as only predecessor the module of length $2 \mod  l$. We visualize this in the following diagram:
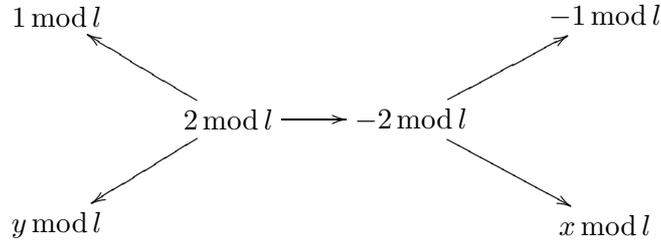
\[\xymatrix{ 1\mod l& & &-1 \mod l\\
&2 \mod l \ar[r]\ar[ld]\ar[lu]& -2 \mod l \ar[ru]\ar[rd]& \\
y \mod l& & & x \mod l}\]

Then by comparing lengths in Auslander-Reiten sequences we get the following equations:

(1) $ 2x+2=0 \mod l$

(2) $x+5= 0 \mod l$

(3) $ 5-y= 0 \mod l$

(4) $ 2y-2=0 \mod l$

We can therefore deduce from (1) and (2) that $l$ divides $8$. Therefore $l=8$ by \ref{d} part (1).

Suppose now that there is an indecomposable module of
length $2$ which is not periodic. By transitivity there is an
indecomposable non-periodic module $M$ of length two in $\theta$.
Then by the equations (1)-(4), $\bar \alpha (M)=3$ and there is an indecomposable
module $N$ which has only $M$ as predecessor. Therefore $M$
appears in the Auslander-Reiten sequence $0\to \tau(N) \to M \to N
\to 0$. This means that $N$ and $\tau(N)$ have length one and are
simple modules. By transitivity $\tau(Q)$ is a simple $B$-module
for any simple $B$-module $Q$. Then $N$ has to be periodic which
is a contradiction. By the equations (1)-(4), $\theta$ does not have an indecomposable module of length $4 \mod l$.
Therefore those modules are periodic by \ref{euc1}.
\end{proof}
We can now exclude tree class $\tilde D_5$ for certain algebras.
\begin{theo}\label{5}
Let $B$ be as in \ref{as}. 
Then $T_s(B)$ does not have a component with tree class $\tilde D_5$.
\end{theo}
\begin{proof} We assume, for a contradiction, that $T_s(B)$ has a component $\theta$ of tree class $\tilde D_5$. Using \ref{euc2,5} and \ref{uni}, we know that $B$ does not have a uniserial module of length $3$, but has a non-periodic indecomposable module of length 3. Therefore $x$ in the proof of \ref{euc2,5} is $3$.
By the proof of \ref{euc2,5} there is an almost split sequence $0 \to \tau(X) \stackrel{f}\to H \stackrel{g} \to X \to 0$ with $H:=\rad P/ \soc P.$ Therefore $l(\tau(X))=l(X)=3$.
Suppose $\tau(X)$ has an indecomposable submodule $U$ of length
$2$. Then $H$ has also an indecomposable submodule $V:=f(U)$ of
length $2$. But then the preimage of $V$ of the canonical
surjection $\rad(P) \to H$ is a submodule of length 3 and is
uniserial which is a contradiction. Therefore $\tau(X)$ has a
quotient $W$ that is indecomposable of length $2$. Let $h: \tau(X)
\to W$ be the canonical surjection. Then by Auslander-Reiten
theory $h$ factors through $f$. Therefore there exists a
surjective map $s:H \to W$. But then $P/ \ker s$ is a uniserial
module of length 3 which is a contradiction. Therefore $T_s(B)$
does not have a component of tree class $\tilde D_5$.
\end{proof}

We define the following automorphisms of $\Z[\tilde D_n]$ as in \cite{F2}.

$$
\alpha(k,i) =
\left\{
\begin{array}{lr}
(k,1),\  i=2\\
(k,2),\  i=1\\
(k,i),\  3\le i
\end{array}
\right.
$$
$$
\beta (k,i) =
\left\{
\begin{array}{lr}
(k,n+1),\  i=n\\
(k,n),\  i=n+1\\
(k,i),\  i\le n-1
\end{array}
\right.
$$
$$
\gamma (k,i) =
\left\{
\begin{array}{lr}
(k,n),\  i=1\\
(k,n+1),\ i=2\\
(k+i-3,n+2-i),\  3\le i \le n-1\\
(k+n-4,1),\  i=n\\
(k+n-4,2),\  i=n+1
\end{array}
\right.
$$
\begin{lemma} \label{aut}  \cite[2.1]{F2} The automorphism group of $\Z[\tilde D_n]$ is given by \[ \{\tau^k \circ \alpha^i\circ \beta^j \circ \gamma ^l |k \in \Z,\  i,j,l \in\{0,1\}\}.\]
\end{lemma}

We describe the action of $G$ on Euclidean components.
\begin{lemma}\label{G-action}
Let $B$ be as in \ref{as}. Let $h\in G$ and suppose $h$ induces an automorphism $A_h:\theta \to \theta , M \mapsto M_h$. Suppose that $B$ has an indecomposable non-periodic module of length 3, if $\theta $ has tree class $\tilde D_n$ for $n>5$. Then $A_h$ is the identity.
\end{lemma}
\begin{proof}
By \ref{euc1} we have that $\theta \cong \Z[\tilde D_n]$ or
$\theta \cong \Z[\tilde A_{1,2}]$. We assume first that $\theta \cong \Z[\tilde D_n]$ with $n>5$.
Suppose $A_h$ is not the identity. By \ref{aut} the automorphisms
of finite order have the form $\tau^k \circ \alpha^i\circ \beta^j
\circ \gamma$ for $k=n/2-2$  or $\alpha^i \circ \beta^j$ with $i,j
\in \{0,1\}$. As $n$ is odd the first possibility cannot occur. Therefore $A_h$ is equal to either $\alpha$, $\beta$ or $\alpha \circ \beta$.

\smallskip

Suppose the map $A_h$ is equal to $\alpha \circ \beta$. Then all
modules with only one predecessors have length $\pm 1 \mod l(P)$.
There exists a non-periodic indecomposable module of length 3 and
by transitivity there is an indecomposable length 3 module $M$ in
$\theta$. As $l\not = 4$ by \ref{d} part (1) we have $\bar \alpha (M)=3$ or $\bar \alpha (M)=2$.
Therefore $M_h\cong M$. This is a contradiction because $M$ has
either a simple top or a simple radical and the map $A_h$ does not
stabilize simple modules.

\smallskip

Assume that $A_h =\beta$. Then $A_h(P/\soc P)=P_h/\soc P_h \not
\cong P/\soc P$. By definition of $S$ and $\phi$, we have $S=\soc
P_g$. Then $S=\soc P_g=\phi(P/\soc P)\not \cong \phi(P_h/\soc
P_h)=\soc P_{hg}=S_{g^{-1}hg}$. Therefore $A_{g^{-1}hg}=\alpha$ as
by the first case no automorphism induced by an element of $G$ is
equal to $\alpha \circ \beta$. But then $A_{hg^{-1}hg}=\alpha\circ
\beta$, which is a contradiction.

\smallskip

Assume now that $A_h=\alpha$. Then $A_h(S)=S_h \not \cong S$. We
have therefore $P/\soc P=\phi^{-1}(S)\not
\cong\phi^{-1}(S_h)=P_{ghg^{-1}}/\soc P_{ghg^{-1}}$. Therefore
$A_{ghg^{-1}}=\beta$ and $A_{hghg^{-1}}=\alpha\circ \beta$, which
is a contradiction.

\smallskip

In the case of $\Z[\tilde A_{1,2}]$ there are no finite order automorphisms
unequal to the identity, so this gives a contradiction as well.
\end{proof}

We describe the non-periodic components more precisely in the
following

\begin{cor}\label{number of comp}
Let $B$ be as in \ref{G-action}. Then $B$ has exactly as many
non-periodic Auslander-Reiten components as there are isomorphism
classes of simple left $B$-modules.
\end{cor}
\begin{proof}By Theorem \ref{euc1} all non-periodic components are
isomorphic and for every non-periodic Auslander-Reiten component
$\Delta$ there exists a $g\in G$ such that $\theta _g = \Delta$. The
component $\theta $ contains a simple module by \ref{simple} and
therefore every non-periodic Auslander-Reiten component contains a
simple module. By transitivity there exists for any simple module
$V$ a non-periodic Auslander-Reiten component $\mathcal{W}$ such that $V$
belongs to $\mathcal{W}$. Suppose there is a non-periodic component which
contains two simple modules $V$ and $V_r$ for some $r\in G$. Then
$r$ induces a non-identity automorphism of finite order on the
component. This is a contradiction to \ref{G-action}.
\end{proof}

We can now prove some necessary conditions for a component of tree class $\tilde D_n$ for $n>5$.
Compare the following Theorem to \cite[4.6]{F2}. We have proved that $n\not =5$. Farnsteiner first shows that $n=5$ and then deduces the other statements from this fact. As this step is wrong,
 we require a different proof.

\begin{theo}\label{euc2}
Let $B$ be as in $\ref{as}.$  Suppose $T_s(B)$ has a
component $\theta$ of tree class $\tilde D_n$, $n>5$.
Suppose $B$ contains an indecomposable non-periodic module of length 3. Then $l(P)=8$ and all modules of
length 4 $\mod l(P)$ are periodic.
\end{theo}
\begin{proof}
The proof follows in two steps.

\smallskip




Let $l:=l(P)$.

Step 1: $l=8$.

\smallskip

Let $M$ be an indecomposable length 3 module in $\theta$. By (c) of \ref{indec}, $\bar \alpha (M)=1$. Suppose $M$ shares a predecessor with the
module of length $1\mod l$. Then the predecessor has length $2 \mod l$ and $6 \mod l=2 \mod l$ which is a contradiction as $l\not =4$. It must therefore share a predecessor with the module of length $-1 \mod l$. This gives us $6= -2 \mod l$ and therefore $l=8$.

\medskip

We know from \ref{d} (2) that the modules with 3 predecessors have length $2
\mod 8$ and $-2 \mod 8$. The modules with one predecessor have therefore length $\pm 1 \mod 8$ or $\pm 3 \mod 8$.

\medskip

Step 2:  The indecomposable modules of length $4 \mod 8$ are periodic.

\smallskip

Let $W$ be a module with one predecessor and length $1 \mod 8$. We take $W$ corresponding to
 $(k,1)$ and use the notation of the proof \ref{indec}. Then
$l(M_3)=2 \mod 8$. Let $\bar W$ be the other module with only
predecessor $M_3$. Then $l(\bar W)= 4x+1$. The module $l(M_4)$ satisfies
$1+l(\bar W)+l(M_4)=4 \mod 8$. Therefore $l(M_4)=2(1-2x) \mod 8$. In the
same way we follow $l(M_5)=2 \mod 8$, $l(M_6)=2(1+2x)\mod 8$, $l(M_7)=2
\mod l$, $l(M_8)=2(1-2x)\mod 8$. The calculation shows that
$l(M_t)=2\mod 8$ if $t$ is odd, $l(M_t)=2(1+2x) \mod 8$ if $t=4m+2$ and
$l(M_t)=2(1-2x) \mod 8$ if $t=4m$ for any $m \in \N$.

Thus modules of length $4 \mod 8 $ in $\theta$ do not have two predecessors. By the remark before Step 2 they do not have one or three predecessors. Therefore no module of length $4\mod 8$ belongs to $\theta$.
As no module of length $4 \mod 8$ appears in $\theta$, they
have to be periodic by (2) of \ref{euc1}.
\end{proof}
Note also that by the proof of \ref{5} $B$ has a uniserial module of length 3.

\smallskip
\section{Application to Auslander-Reiten Components of enveloping algebras of restricted $p$ Lie algebras}

Let $L$ be a finite-dimensional restricted $p$-Lie algebra and
$\chi$ a linear form on $L$. We denote by $u(L, \chi)$ the
universal enveloping of $(L, \chi)$. If $\chi=0$ we set $u(L,
\chi)=u(L)$.

\medskip

We denote by $G(L)$ the set of group-like elements of the dual
Hopf algebra $u(L)^*.$ The set of group-like elements are the
homomorphisms of $u(L)$. The comultiplication on $u(L)$ induces an
algebra homomorphism $\Delta: u(L,\chi) \to u(L)\otimes u(L,\chi),
x \mapsto x\otimes 1+1\otimes x$ for all $x \in L$.  We denote $\Delta(u)=u_1\otimes u_2$ for $u \in u(L, \chi)$.
This defines a left $u(L)$-comodule algebra structure and right $u(L)^*$-module
algebra structure on $u(L,\chi)$. Therefore $G(L)$ acts on the
automorphism group of $u(L,\chi)$ via $(g\cdot \psi) (u)= \psi(u
\cdot g)=g(u_1)\psi(u_2)$ for all $\psi \in \aut(u(L,\chi))$,
$g\in G(L)$ and $u\in u(L, \chi)$. We embed $G(L) $ into
$\aut(u(L,\chi))$ via the injective group homomorphism $f: G(L)
\to \aut(u(L,\chi)), w\mapsto w\cdot \id_{u(L,\chi)}$. For an
$u(L,\chi)$-module $M$ and $w\in G(L)$ we denote by $M_w$ the
twisted module $M_{f(w)}$. Note that every element of
$G(L)\setminus \{1\}$ has order $p$.

\medskip

By \cite[1.2]{FS1} the Nakayama automorphism of $u(L,\chi)$ has
order $1$ or $p$ and all modules of complexity one are
$2$-periodic by \cite[2.5]{F1}. Furthermore $u(L,\chi)$ has a
non-periodic indecomposable module of order $3$ by \cite[4.5]{F2}.
 Therefore the assumptions of \ref{as} are satisfied  for $G(L)$-transitive blocks or $G$-transitive blocks, where $G$ is a finite subgroup of $\aut(u(L,\chi))$. The next corollary follows directly from \ref{5}.
\begin{cor}
Let $B\subset u(L,\chi)$ be a $G$-transitive block, then $T_s(B)$ does not have a component of tree class $\tilde D_5$.
\end{cor}
More generally we have
\begin{lemma}\label{lie} Let $B\subset u(L,\chi)$ be a $G$-transitive block and let $S$ be a simple module in $B$, then $T_s(B)$ admits an Euclidean component only if $p=2$ or $\dim S=0 \mod p$.
\end{lemma}
\begin{proof} By \ref{euc1} and \ref{euc2} all indecomposable modules of length two or of length four are periodic. By \cite[2.5]{F1} all periodic
 indecomposable modules have dimension $0 \mod p$. As $B$ is $G$-transitive all simple modules in $B$ have the same dimension. Therefore $ 2\dim S= 0 \mod p$.

\end{proof}
As a $G$-transitive principal block of $u(L,\chi)$ has only one
dimensional simples, we get the following corollary immediately
from the preceding Lemma.

\begin{cor}Let $B\subset u(L,\chi)$ be the principal block, and assume $B$ is $G$-transitive, then $T_s(B)$
admits an Euclidean component only if $p=2$.
\end{cor}
We call a block $B$ primary if it only contains one isomorphism type of simple modules. Note that lemma \cite[4.7]{F2} remains true for primary blocks of
 $u(L, \chi)$ with the additional assumption that all indecomposable modules of length 2 are periodic.

\medskip

We remind of the definition of supersolvable Lie algebras
\begin{defi}\cite[I]{FS}Let $(L^i)_{i\in N}$ with $L^i=[L^{i-1},L]$ and $L^0=L$ be a sequence of ideals in $L$.
Then $L$ is nilpotent if there is an $n \in \N$ such that $L^n=0$.
The sequence $(L^{(i)})_{i\in N}$ with
$L^{(i)}=[L^{(i-1)},L^{(i-1)}]$ and $L^{(0)}=L$ is the derived
series. We call $L$ solvable if there is an $n \in \N$ such that
$L^{(n)}=0$ and $L$ supersolvable if $L^1$ is nilpotent.
\end{defi}
 Using
the fact that projective modules of restricted universal
enveloping algebras of supersolvable Lie algebras have $p$-power
length by  \cite[2.10]{F3}, the result of \cite[4.1]{F2} remains
true by applying \ref{euc2}.
\begin{lemma}
Let $L$ be a supersolvable finite-dimensional restricted $p$-Lie algebra and $p>2$. Then $T_s(u(L,\chi))$ does not have a component of Euclidean tree class.

\end{lemma}This result cannot be extended to $p=2$ as the following example shows.

\begin{example}Let $A=k[x,y]/(x^2, y^2)$ be the Kronecker algebra. Then $A \cong u(L)$ where $L=\Sp \{x,y\}$ is the restricted
$2$-Lie algebra given by $[x,y]=0$ and $x^{[2]}=y^{[2]}=0$. Then
$L$ is supersolvable and the component containing the trivial
module $k$ is isomorphic to $\Z[\tilde A_{1,2}]$. This is well known, see for example \cite[2.3]{E}.
\end{example}

\section{Euclidean components of smash products}
The goal of this section is to determine conditions so that the
smash product of a basic simple algebra and a semi-simple
commutative group algebra have an Auslander-Reiten component of
Euclidean tree class. We assume that $k$ is algebraically closed.

\smallskip

We start by describing the simple and indecomposable projective modules of
certain smash products.

\begin{lemma}\label{struc} Let $\Gamma$ be a local and basic algebra with simple module $S$ and let $G$ be a finite group such that $G < \aut(\Gamma)$. Let $\{e_1,\ldots,  e_m\}$
 be a full set of primitive orthogonal idempotents in $kG$, let $\bigoplus_{i=1}^m P_i$ be
 a decomposition of $kG$ into projective indecomposable $kG$-modules $P_i:=kGe_i$ and let $S_i:= \soc P_i$ for $1 \le i \le m $.

 Then for every simple $\Gamma \rtimes kG$-module $V$ there exists an
$1\le i\le m$ such that $V \cong S\otimes S_i $. A complete set of
primitive orthogonal idempotents of $\Gamma\rtimes kG$ is given by
$\{1\rtimes e_i| 1\le i \le m\} $ and $\Gamma\rtimes kG$ has a
decomposition $\bigoplus_{i=1}^m \Gamma \rtimes P_i$ into
projective indecomposable modules $\Gamma \rtimes P_i.$
\end{lemma}
\begin{proof}As $g$ induces an automorphism on $\Gamma$ for all $g\in G$, we have $G(J(\Gamma))=J(\Gamma)$ and $J(kG)\Gamma \subset  J(\Gamma)$.  Therefore
$J(\Gamma)\rtimes kG+ \Gamma \rtimes J(kG) \subset J(\Gamma
\rtimes kG)$. Furthermore $\Gamma \rtimes kG/ (J(\Gamma)\rtimes
kG+ \Gamma \rtimes J(kG)) \cong \Gamma/J(\Gamma) \otimes kG/J(kG)
\cong \bigoplus_{i=1}^m S \otimes S_i$ which is semi-simple. This
proves  $J(\Gamma)\rtimes kG+ \Gamma \rtimes J(kG) = J(\Gamma
\rtimes kG)$ and all simples are
 given by $S \otimes S_i$. Clearly $\{1\rtimes e_i| 1\le i \le m\}$ is a set of orthogonal idempotents and $\bigoplus_{i=1}^m\Gamma \rtimes P_i$ is a decomposition of $\Gamma \rtimes kG$
 into projective modules  $\Gamma \rtimes P_i=(\Gamma \rtimes kG )(1 \rtimes e_i)$. The projective modules are indecomposable as $ \soc(\Gamma \rtimes P_i)= S \otimes S_i$ is simple
 and therefore $\{1 \rtimes e_i| 1\le i \le m\} $ is a complete set of primitive idempotents.
\end{proof}

>From now on let $\Gamma$ be a basic local algebra with simple
module $S$. Let $G$ be an abelian group such $kG$ is semi-simple,
and $G$ is a subgroup of $\aut(\Gamma).$ Then the smash product
$R:=\Gamma \rtimes kG$ is well defined.

\medskip

By Gabriel's lemma \cite[4.1.7]{B}, there exists a quiver $Q$ such
that $\Gamma\cong k Q/ I$ for an admissable ideal $I \subset kQ$.
As $G$ is abelian and $kG$ semi-simple, the set of irreducible
characters of $kG$ forms a multiplicative group isomorphic to $G$.
We index the characters by elements of $G$ via a fixed isomorphism
and index the primitive orthogonal idempotents by the same group
element as its corresponding character.  So let $\{\chi_g|g\in
G\}$ be the set of irreducible characters and $\{e_g|g \in G\}$
the set of primitive orthogonal idempotents, such that
$he_g=\chi_g(h)e_g$ for all $g,h \in G$. Suppose $G \le
\aut(\Gamma)$. Then $kG$ acts on $J(\Gamma)$ and $J^2(\Gamma)$. As
$kG$ is semi-simple, $J(\Gamma)/J^2(\Gamma)$ split as a direct sum
of one-dimensional $kG$-modules. Let $\alpha_1, \ldots, \alpha_m$
be the simultaneous eigenvectors of the action of $G$ on
$J(\Gamma)/J^2(\Gamma)$. Let $\chi_{n_i}$, $n_i \in G$,
$i=1,\cdots , m$ be the corresponding irreducible characters. By
\ref{struc} we know that $\Gamma \rtimes kG$ is a basic algebra
with projective indecomposable modules $\Gamma \rtimes ke_g$ for
$g \in G$ which have simple quotients $S\rtimes ke_g$. We have the
following presentation of $\Gamma \rtimes kG$. Take the quiver
where vertices are labelled by $1\rtimes e_g$ and where arrows are
$\alpha_i \rtimes e_g$. Note that
\[(\alpha_i \rtimes e_h)(\alpha_j \rtimes e_g)=(\alpha_i \alpha_j)\rtimes \chi_{n_jg}(e_h)e_g=(\chi_{n_{j}g},\chi_h)(\alpha_i \alpha_j \rtimes e_g)\] where $(-,-)$ is the usual inner product of characters. Therefore the arrow $\alpha_i
\rtimes e_g$ ends in $1\rtimes e_g$ and starts in $1\rtimes e_q$
with $q=g n_i$.  We can obtain the relations that generate $T$ via
the relations that generate $I$ in $\Gamma$.

\medskip

Note that the construction of $W$ coincides with the Mc Kay quiver (see \cite[2]{M} for the definition) where $V:=J(\Gamma)/J^2(\Gamma)$.

\smallskip

We will illustrate this construction on a small example.
\begin{example}
Let $\Gamma=k[x,y]/\langle x^2,y^2 \rangle$ the Kronecker algebra
and let $G=\langle g \rangle$ be a cyclic group of order 3.

Then $\Gamma \cong kQ/I$ with $$Q= \xymatrix{\bullet\ar@(dl,ul)^{x}\ar@(dr,ur)_{y}}$$ and $I=\langle
x^2,y^2, xy-yx \rangle$. The algebra $\Gamma $ is a $kG$-module algebra
via the action $gx=q^{-1}x$, $gy=q y$ and $gxy=xy$ for a primitive
third root of unity $q$. We label the character $\chi$ with $\chi(g)=q$ as $\chi=\chi_g$.
Then $n_x=g^2$ and $n_y=g$. Let $e_1$,
$e_g$ and $e_{g^2}$ denote the primitive idempotents in $G$ such
that $ge_1=e_1$, $ge_g=qe_g$ and $ge_{g^2}=q^{-1}e_{g^2}$. Then
the primitive idempotents  are given by $1\rtimes e_i$.
 We construct the quiver $W$ as described in the previous example.
\[\xymatrix{ 1\rtimes e_{g^2} \ar[rr]_{y\rtimes e_g} \ar@/_/[dr]_{x\rtimes e_1}&
 & 1\rtimes e_g\ar@/_/[ll]_{x\rtimes e_{g^2}}\ar[dl]_{y\rtimes e_1}\\
&1\rtimes e_1 \ar[ul]_{y\rtimes e_{g^2}} \ar@/_/[ru]_{x\rtimes
e_g}& }\] The relations are given by \[T:= \langle (y\rtimes e_i)(
y \rtimes e_j), (x\rtimes e_i)( x\rtimes e_j),\  (x\rtimes e_{gj})( y
\rtimes e_j)-(y \rtimes e_{g^2j})( x \rtimes e_j)| i,j \in \{1,g,
g^{2} \} \rangle.\]
 By the previous example $\Gamma \rtimes kG \cong kW/T$.
 \end{example}

 We have that $R$ is $G$-transitive via $g (a\rtimes h)= a
\rtimes \chi_{g^{-1}}(h) h$ or $g(a \rtimes e_h)=a\rtimes e_{gh}$
for all $a\in \Gamma$ and $g,h \in G$. With this action $G$ is a
subgroup of $\aut (R)$. Note that $1\rtimes G\cong G$ is a
subgroup of $R$ and $k\rtimes G\cong kG$ is a subalgebra of $R$.
We first define the following notation. Let $C$ be an
$R$-module, then $C$ is a $kG$-module via $g \cdot c:=
(1\rtimes g )c$ for all $c \in C$ and $g\in G$. We denote by $C_g$
the $R$-module with $(a\rtimes h)*c:=\chi_{g^{-1}}(h)(a\rtimes
h)c$ for all $c\in C$, $g,h\in G$ and $a\in \Gamma$. If $C$ is a
$\Gamma$-module, we denote by $C_g$ the module with $t*c:=g(t)c$
for all $c\in C$ and $t\in \Gamma$.

If $C$ is a $\Gamma$ or an
$R$-module, then we set $S(C):=\{g\in G|C_g \cong C \}$ with the
respective action of $G$ on $R$ and on $\Gamma$-modules. Let
$T(C)$ be a transversal of $G/ S(C)$.

\smallskip

We determine how $R$-modules or $\Gamma$-modules decompose if
restricted to $\Gamma$ or respectively lifted to
$R$.

\begin{lemma}\label{ind}

(1) Let $M$ be an $R$-module. Then $(M_{\Gamma})^R\cong \bigoplus_{g\in G}
M_g$.

(2) Let $N$ be a $\Gamma$-module, then $N^R_{\Gamma}\cong
\bigoplus_{g\in G} N_g$ where $N_g$ denotes the twist of $N$ by
the element $g \in \aut(\Gamma)$.

(3) Let $M$ be an indecomposable $R$-module and $N$ an
indecomposable $\Gamma$-module such that $N| (M_{\Gamma})$. Then
$M_{\Gamma}=q \bigoplus_{g\in T(N)}N_g$, $N^R= n \bigoplus_{g\in
T(M)}M_g$ and $qn|T(N)||T(M)|=|G|$ for some $n,m \in \N$.

\end{lemma}
\begin{proof}(1) Let $\psi_g: M_g \to (M_{\Gamma})^R, m\mapsto  |G|^{-1}\sum_{l \in G}
\chi_g (l)(1\rtimes l\otimes l^{-1}m)$ and let $\phi_g:
(M_{\Gamma})^R \to M_g, r\otimes m  \mapsto g(r)m$ for all $r \in
R$, $m \in M$ and $g\in G$. Let $a\in\Gamma$ and $h\in G$. Then
\begin{eqnarray*} \psi_g((a\rtimes
h)*m)&=&\psi_g(\chi_{g^{-1}}(h)(a\rtimes h) m )\\
&=&|G|^{-1}\sum_{l
\in G} \chi_{g^{-1}}(h) \chi_g (l)(1\rtimes l\otimes
l^{-1}(a\rtimes h)m)\\
&=& |G|^{-1}\sum_{l \in G}
\chi_{g^{-1}}(h)\chi_g (l)(1\rtimes l)(l^{-1}(a)\rtimes 1\otimes
l^{-1}h m)\\
&=&|G|^{-1}(a\rtimes 1)\sum_{l \in G}
\chi_{g^{-1}}(h)\chi_g (l)(1\rtimes l \otimes
l^{-1}h m)\\
&=&(a\rtimes h)|G|^{-1}\sum_{s \in G}\chi_g
(s)(1\rtimes s \otimes s^{-1} m)\\
&=&(a\rtimes h)\psi_g(m),
\end{eqnarray*} by substituting $l^{-1}h=s^{-1}$, and therefore $\psi_g$ is an $R$-module
homomorphism.

It is clear that $\phi_g$ is an $R$-module homomorphisms and
$\phi_g \circ \psi_g= \id_{M_g}$. Let $\{m_1,\ldots, m_n\}$ be a
$k$-basis of $M$. A basis of $(M_{\Gamma})^R$ is given by \[\{(1
\rtimes l )\otimes m_i | \mbox{ for } 1\le i\le n \mbox{ and }l
\in G \}. \] Using this basis we have $\psi_g (m)= \psi_h(\bar m)$
for some $m,\bar m \in M$ if and only if $\chi_g(l)m=\chi_h(l)\bar
m$ for all $l\in G$. Therefore $\psi_g (M_g)\cap \psi_h(M_h) =
{0}$ for $g \not = h$. Finally by comparing dimensions we have
$(M_{\Gamma})^R\cong \bigoplus_{g\in G} M_g$.

\smallskip

(2) We have $N^R_{\Gamma}= \bigoplus_{g\in G} 1\rtimes g \otimes
N$. Furthermore $1\rtimes g \otimes N \cong N_{g^{-1}}$ as
$\Gamma$-module, which proves the statement.

\smallskip

(3) Suppose $Q$ is an indecomposable module with
$Q|(M_{\Gamma})$, then $Q^R$ and $N^R $ are direct summands of
$(M_{\Gamma})^R=\bigoplus_{g\in G} M_g$. As $(M_g)_{\Gamma}\cong
M_{\Gamma}$ for all $g\in G$, we have that
$N|(Q^R_{\Gamma})=\bigoplus_{g \in G} Q_g$. Therefore $Q\cong
N_g$ for some $g\in G$. Furthermore $(M_{\Gamma})_g \cong
M_{\Gamma}$ via the $\Gamma$-module isomorphism
$\psi:M_{\Gamma}\to (M_{\Gamma})_g, m \mapsto gm$ for all $m \in
M$ and $g \in G$. Therefore $M_{\Gamma}$ is $G$-invariant. This
proves the first identity.

By the first identity, we know that all
indecomposable direct summands of $N^R$ are isomorphic to $M_g$ for
some $g\in G$. Note that $N^R$ is $G$-invariant via the $R$-module
isomorphism $\phi:N^R\to (N^R)_g, r\otimes n \mapsto g(r)\otimes
n$ for all $g\in G$, $r\in R$ and $n\in N$. This map is well
defined as $G$ acts on $\Gamma \rtimes 1 \subset R$ as the identity.
Therefore the second identity holds.

Finally we compare the multiplicity of $N$ as a direct summand of $N^R_{\Gamma}$.
The first and second identity of $(3)$ give a multiplicity
of $n|T(M)|q$ and (2) gives multiplicity $|S(N)|$.
\end{proof}
By standard arguments we deduce the next two lemmas.
\begin{lemma}
Every $R$-module $M$ is relatively
$\Gamma$-projective.
\end{lemma}
\begin{proof}
 Suppose $A$, $B$ are $R$-modules and $h : A \to B$,
$f: M \to B$ are $R$-module homomorphisms. Suppose there is a
$\Gamma$-module homomorphism $v: M_{\Gamma} \to A$ such that
$h\circ v =f$. Then $\bar v: M\to A, m \mapsto |G|^{-1}\sum_{g\in
G}g v(g^{-1}m)$ is an $R$-module homomorphism. This can be seen as
follows: let $t\in \Gamma$, $h\in G$, then \begin{eqnarray*} \bar
v((t\rtimes h)m)&=&|G|^{-1}\sum_{g\in G}g v(g^{-1}(t\rtimes
h)m)\\
&=& |G|^{-1}\sum_{g\in G}g v(g^{-1}(t)\rtimes
g^{-1}h)m)\\
&=& |G|^{-1}(t\rtimes 1) \sum_{g\in G}g v(g^{-1}h
m)\\
&=&(t\rtimes h) |G|^{-1}\sum_{s\in G}s v(s^{-1}m)\\
&=& (t\rtimes h)\bar v(m) ,\end{eqnarray*} if we substitute
$s^{-1}=g^{-1}h$. Furthermore $\bar v$ satisfies $h \circ \bar v=
f$.
\end{proof}
\begin{lemma}[Frobenius reciprocity]\label{frob}
Let $V$ be a $\Gamma $-module and $M$ an $R$-module. Then there is
a bijection of vector spaces between $\Hom_{\Gamma}(V,
M_{\Gamma})$ and $\Hom_R(V^R, M)$.
\end{lemma}
\begin{proof}The bijection is given by $\psi : \Hom_{\Gamma}(V,
M_{\Gamma}) \to \Hom_R(V^R, M)$ where $\psi(f) (r \otimes v ) = rf(v)$
 and $\phi: \Hom_R(V^R, M) \to \Hom_{\Gamma}(V,M_{\Gamma})$ with $\phi(g)(v)=
g(1 \otimes v)$ for all $r \in R$, $v\in V$, $f \in
\Hom_{\Gamma}(V,M_{\Gamma})$ and $g \in  \Hom_R(V^R, M)$.
\end{proof}

Let $G(A)$ denote the free abelian group of an algebra $A$ with
free generators $[V_i]$, the representatives of the isomorphism
classes of all indecomposable $A$-modules $V_i$. If $M = \oplus a_i V_i$ where the $a_i \geq 0$ then we write
$[M] := \sum a_i[V_i]$. We denote by
$(-,-)_A: G(A) \times G(A) \to k$ the bilinear form $\dim_k \Hom_A
(-,-)$.

Let $Q: 0\to B\to C\to D\to 0$ be an exact sequence. Then we set
$[[Q]]:=[B]+[D]-[C]\in G(A)$. Let $A(V_i)$ denote the
Auslander-Reiten sequence starting in $V_i$ for $V_i$
non-projective. Furthermore we set $X_i:=[[A(V_i)]]$ for $V_i$
non-projective and $X_i:= [V_i]-[\rad (V_i)]$ if $V_i$ is
projective.

\smallskip

The first part of the next Theorem is the general version of \cite[3.4]{BP},
 that was only proven for group algebras.

\begin{theo}\label{dual} Let $A$ be a finite-dimensional algebra over an algebraically closed field $k$.

(1)We have that $([V_i], X_j)=\delta_{i,j} $. Therefore $(-,-)_A$ is non-degenerate. Furthermore $([V_i],[[E]])\ge 0$ for any exact
sequence $E$.

(2) Suppose $Q:=0\to C  \to B \to V_j\to 0$ is an exact non-split
sequence with $[[Q]] \not =X_j$, then there is a $V_i$ with $j
\not = i$ such that $([V_i],[[Q]])\ge 1$.
\end{theo}
\begin{proof}

(1) \ Take the almost split sequence
$$0 \to \tau(V_j)\stackrel{l}  \to M_j \stackrel{s} \to V_j \to 0
$$
this gives an exact sequence
$$0 \to {\rm Hom}_A(V_i, \tau V_j) \to {\rm Hom}_A(V_i, M_j)
\stackrel{\psi}\to {\rm Hom}_A(V_i, V_j).
$$
If $i\neq j$ then by the Auslander-Reiten property, the map $\psi$ is
onto and it follows that $([V_i], X_j)=0$.
If $i=j$ then by the Auslander-Reiten property, ${\rm Im}(\psi)$ is the
radical of ${\rm End}(V_i)$. Therefore
$$([V_i], X_i) =  (V_i, \tau(V_i)) + (V_i, V_i)-(V_i, M_i) =
 (V_i, V_i) - \dim {\rm Im}(\psi)
$$
and this is equal to
$$ \dim ({\rm End}(V_i)/{\rm rad}({\rm End}(V_i)).
$$
Since $k$ is algebraically closed and $V_i$ is indecomposable, this number is equal to $1$.

Let $E: =0 \to S \to T \to U\to 0$ be an exact sequence. Then $0 \to \Hom_A(V_i, S) \to \Hom_A(V_i, T) \to \Hom_A(V_i,U)$ is exact. Therefore $([V_i],[[E]]) \ge 0$.

\medskip

(2) Let $ Q:= 0\to C \stackrel{\delta} \to B \stackrel{\sigma} \to V_j\to 0$ be an exact sequence.
Suppose that $[[Q]] \not =[[A(V_j)]]$.
 Then we get the following commutative diagram

\[\xymatrix{0 \ar[r] & C\ar[r]_{\delta} \ar[d]_g& B \ar[r]_{\sigma} \ar[d]_h&
V_j\ar[r]\ar[d]_{\id} & 0\\0 \ar[r] & \tau(V_j) \ar[r]_l& M_j \ar[r]_s&
V_j\ar[r] & 0},\] where the existence of $h$ follows from the Auslander-Reiten property since
 the map $\sigma : B\to V_j$ is non-split; and $g$ is the restriction
of $h$ to $C$.

This diagram induces a short exact sequence \[Z:=0\to C
\stackrel{ \binom{\delta}{g}} \to  B\oplus \tau(V_j) \stackrel{(h,l)} \to M_j\to 0.\] Suppose
that this sequence is split. Then there is a map
$\binom{f_1}{f_2}: M_j\to B\oplus \tau (V_i)$ such that $h\circ f_1 +
l\circ f_2=\id_{M_j}$. Then $s\circ hf_1=s$. As $s$ is minimal right
almost split, we have that $h f_1$ is an automorphism.
We also have $\sigma f_1 = s$, so let $g_1: \tau(V_j)\to C$ such that
$\delta g_1 =  f_1 l$. Then also $gg_1$ is an automorphism. So we have
$B= M_j\oplus {\rm ker}(h)$ and $C = \tau(V_j)\oplus {\rm ker}(g)$. By the Snake
Lemma, ${\rm Ker}(g)\cong {\rm Ker}(h)$.
But then $[[Q]]=[[A(V_j)]]$. So we have a contradiction, therefore
$Z$ is non-split. Then we have, working in $G(A)$, that $[[Q]] =
[[Z]] + [[A(V_j)]]$ and hence we have $([M_j],[ [Q]])= ([M_j],
[[Z]]+[[A(V_j)]])= ([M_j],[[Z]])\ge 1$ as the image of the map
$\Hom(M_j,M_j) \to \Ext^1(M_j, C)$ induced by $Z$ contains the
non-split sequence $Z$. As $(-,-)_A$ is biadditive and $V_j$ is
not a summand of $M_j$, the second statement is proven.
\end{proof}
We can write the element $[[Q]]\in G(A)$ for any exact sequence
$Q$ ending in $W$ as a sum in $G(A)$ of $[[Q_1]]$ and $[[Q_2]]$
for two short exact sequences $Q_1$ and $Q_2$ ending in direct
summands of $W$.
\begin{lemma}\label{sum} Suppose that $Q:= 0\to U \to V \stackrel{\pi} \to  W \to 0$
is an exact sequence and $W= W_1\oplus W_2$ for two non-trivial
$A$-modules $W_1$ and $W_2$. Then there is an exact sequence $Q_1$ ending in $W_1$ and
an exact sequence $Q_2$ ending in $W_2$ such that $[[Q]]=[[Q_1]]+[[Q_2]]$.
\end{lemma}
\begin{proof}
Let $p_i:W\to W_i$ be the natural
projection for $i=1,2$. Then
\[Q_1:=0\to \pi^{-1}(W_1) \to V \stackrel{p_2 \pi} \to  W_2 \to 0\] and
\[Q_2:=0\to U \to \pi^{-1}(W_1) \stackrel{p_1 \pi} \to  W_1 \to 0\] are exact sequences and
$[[Q]]= [[Q_1]]+[[Q_2]]$ in $G(A)$.
\end{proof}
Furthermore we have $([V],[[Q]])\ge
([V],[[Q_1]])$ for any module $V$ as $([V],[[E]])\ge 0$ for any exact
sequence $E$ by \ref{dual}(1).

\medskip

We can prove the next result.

\begin{theo}\label{aus restrict}
Let $M$ be an indecomposable $R$-module and $C$ an indecomposable
$\Gamma $-module, such that $M$ is a direct summand of $C^R$ with
multiplicity $n$. Then $[[A(M)_{\Gamma}]]= n\sum_{g \in T(C)}[[A(C_g)]].$

\end{theo}

\begin{proof}We first show that $([V],[[A(M)_{\Gamma}]])-n\sum_{g \in T(C)}([V],[[A(C_g)]])=0$ for all
indecomposable $\Gamma$-modules $V$. Using Frobenius reciprocity
\ref{frob}, we have $([V], [[A(M)_{\Gamma}]])=([V^R], [[A(M)]])$
which is equal to the multiplicity of $M$ as a direct summand in
$V^R$. By Lemma \ref{ind} we have that $(M_{\Gamma}) |
(C^R)_{\Gamma}= \oplus_{g\in G}C_g$. Therefore $M$ is a direct
summand of $V^R$ if and only if $V \cong C_g$ for some $g\in G$.
But in this case $V^R=C^R$ and therefore the multiplicity of $M$
as a direct summand of $V^R$ is $n$. It remains to show that
$[[A(M)_{\Gamma}]]$ is a linear combination of Auslander-Reiten
sequences $X_i$. We have $M_{\Gamma}=q\sum_{g\in T(C)} C_g$. By \ref{sum} we know that
$[[A(M)_{\Gamma}]]$ can be written as sum of $[[Q^i_g]]$ where the $Q^i_g$ are
exact sequences starting in $C_g$ for $g\in G$ for $1\le i \le q$. Suppose one of them is
non-split and not an Auslander-Reiten sequence $X_i$. By the second part of
\ref{dual} there exists an indecomposable direct summand $L$ of the
middle term of some $A(C_g)$ such that $([L],[[A(M)_{\Gamma}]])\ge 1$.
There is no irreducible map from $C_l$ to $C_h$ for any $l,h\in G$
as both modules have the same dimension. Therefore $L\not \cong
C_l$ for all $l\in G$. By \ref{dual} we have $([L], [[A(C_l)]])=0$
for all $l\in G$, which is a contradiction to the first part.
\end{proof}

Clearly if $A(M) $ is an Auslander-Reiten sequence and $M$ a
non-projective indecomposable module, then $\tau(M)$, $M$ and the middle term of
$A(M)$ have no direct summand in common. The same is true for the
restriction of the Auslander-Reiten sequence to $\Gamma$.
\begin{lemma}\label{cancel}
Let $M$ be an indecomposable $R$-module and $A(M):=0 \to \tau (M)
\to X \to M\to 0$. Then the pair $\tau(M)_{\Gamma}$ and
$X_{\Gamma}$ and the pair $M_{\Gamma}$ and $X_{\Gamma}$ have no
direct summand in common. If $M_g \not \cong \tau(M)$, then
$\tau(M)_{\Gamma}$ and $M_{\Gamma}$ have no direct summand in
common.
\end{lemma}
\begin{proof}
Suppose there is an indecomposable $\Gamma$-module $Q$, such that
$Q|(M_{\Gamma})$ and $Q| (X_{\Gamma})$. By \ref{ind} there exists
an indecomposable direct summand $E$ of $X$ and a $g \in G$ such
that $E \cong M_g$. As $M$ and $ M_g$ have the same dimension
there is no irreducible map from $E$ to $M$ which is a
contradiction. By an analogous argument it is clear that
$\tau(M)_{\Gamma}$ and $X_{\Gamma}$ have no direct summand in
common. Suppose know that $M_{\Gamma}$ and $\tau(M)_{\Gamma}$
have a common direct summand. Then there exists a $g \in G$ such
that $\tau(M)\cong M_g$, which is a contradiction.
\end{proof}
This gives the next
\begin{cor}\label{restrict aus}
Let $M$ be an indecomposable, non-projective $R$-module.
Let $C$ be an indecomposable direct summand of $M_{\Gamma}$ with multiplicity $n$.
Then $M$ is a direct summand of  $C^R$ with multiplicity $n$.
Furthermore if $N$ is the middle term of $A(M)$ and $Q$ the middle term of $A(C)$,
then $N_{\Gamma}=n\bigoplus_{g\in T(C)}Q_g$ and $\tau(M)_{\Gamma}=n\bigoplus_{g\in
T(C)}\tau(C)_g$.
\end{cor}
\begin{proof}
Let $A(M):= 0\to \tau(M)\to \bigoplus_{1\le i\le t} d_i N_i \to M \to 0$ and $A(C):= 0\to \tau(C) \to \bigoplus_{1\le i\le s}f_i Q_i \to C \to 0$ for $Q_i$, $N_i$ indecomposable and $d_i, f_i \in \N$. We set $Q:= \bigoplus_{1\le i\le s}f_i Q_i$ and $N:= \bigoplus_{1\le i\le t} d_i N_i$.
Then $d_iN^i_{\Gamma}= b_i \sum_{g \in T(E^i)}E^i_g $, $\tau(M)_{\Gamma}=a \sum_{g \in T(L)}L_g$ and $M_{\Gamma}=q \sum_{g \in T(C)}C_g$ for some  indecomposable $L, E^i \in \Gamma$-mod and $a,q,b_i \in \N$. Then \[ [[A(M)_{\Gamma}]]=q \sum_{g \in T(C)}[C_g]+a \sum_{g \in T(L)}[L_g]-\sum_{1 \le i \le t} b_i \sum_{g \in T(E^i)}[E^i_g].\]  By \ref{aus restrict} we also have \[ [[A(M)_{\Gamma}]]= n \sum_{g \in T(C)}[C_g]+ n\sum_{g \in T(C)}[\tau(C)_g] -n \sum_{1\le i \le s} f_i \sum_{g \in T(C)}[Q^i_g].  \]

We assume that $M_g \not \cong \tau(M)$. Then by \ref{cancel} the set $\{[C_g], g \in T(C)\} \cup_{1\le i \le t} \{ [E^i_g], g \in T(E^i)\}\cup  \{ [L_g] , g \in T(L) \} $ is linearly independent. Similarly the set $\{[C_g], g \in T(C)\}\cup_{1\le i \le s} \{ [Q^i_g], g \in T(Q^i)\}\cup  \{ [\tau(C)_g] , g \in T(\tau(C)) \}$ is linearly independent. We can see this as follows: if $\tau(C) \cong C_h$ for some $h \in G$, then $\sum_{g \in T(C)}\tau(C)_g=  \sum_{g \in T(C)}C_g$ which is a contradiction, as the element $[L]$ would not appear as a summand of $[[A(M)_{\Gamma}]]$ in the second presentation. Also $Q^i_h \not \cong \tau(C) $ and $Q^i_h \not \cong C$ for some $h \in G$ because there is no irreducible map between elements of the same dimension. We compare now the two presentations and use the linear independency  of the indecomposable $\Gamma$-modules in $G(\Gamma)$. Then $q=n$, $\tau(M)_{\Gamma}= n\sum_{g \in T(C)}\tau(C)_g$ and $N_{\Gamma}=n \sum_{g \in T(C)}Q_g$.

\smallskip

Assume now that $M_g \cong \tau(M)$. Then $\tau(M)_{\Gamma}=q \sum_{g \in T(C)}C_g$. Therefore \[ [[A(M)_{\Gamma}]]=2q \sum_{g \in T(C)}[C_g]- \sum_{1 \le i \le t} b_i \sum_{g \in T(E^i)}[E^i_g] .\] By \ref{cancel} we have that the set $\{[C_g], g \in T(C)\} \cup_{1\le i \le t} \{[ E^i_g], g \in T(E^i)\}$ is linearly independent in $G(\Gamma)$. We have $\tau(C)=C_h$ for some $h \in G$ by comparing summands in the two presentations of $[[A(M)_{\Gamma}]]$. Then $\sum_{g \in T(C)}\tau(C)_g=\sum_{g \in T(C) } C_g$. Therefore $q=n$, $\tau(M)_{\Gamma}= n\sum_{g \in T(C)}\tau(C)_g$ and $N_{\Gamma}=n \sum_{g \in T(C)}Q_g$.
\end{proof}
We can now investigate the relation between periodic
$R$-modules and periodic $\Gamma$-modules.
\begin{lemma}[periodic modules]\label{periodic}
The indecomposable $R$-module $M$ is periodic if and only if
$M_{\Gamma}$ contains a periodic direct summand.
\end{lemma}
\begin{proof}
Let $Q$ be an indecomposable direct summand of $M_{\Gamma}.$ Then
$M_{\Gamma}=n \bigoplus_{g\in T(Q)} Q_g$ and $\tau(M)_{\Gamma}=n
\bigoplus_{g\in T(Q)} \tau (Q)_g$ by \ref{restrict aus}. Suppose
now that $Q$ is $\tau$-periodic with period $m$. Then $Q_g$ is
$\tau$-periodic with period $m$ and $\tau^m(M)_{\Gamma}\cong
M_{\Gamma}$ as by \ref{restrict aus} $\tau$ and the restriction to $\Gamma$
commute. This implies that
$M^g\cong \tau^m(M)$.  As
$G$ has finite order, $M$ is periodic. Similarly, if $M$ is
$\tau$-periodic with period $m$, we have $n \bigoplus_{g\in T(Q)}
Q_g=M_{\Gamma}= \tau^m(M)_{\Gamma}=n \bigoplus_{g\in T(Q)}
\tau^m (Q)_g$. Therefore $\tau^m(Q)\cong Q_g$ for some $g\in G$.
As $G$ has finite order, $Q$ is $\tau$-periodic.
\end{proof}

Next, we summarize the properties that we need for the following
theorems.

\begin{as}\label{as2}
In the following we assume that $\Gamma$ and $R$ are Frobenius algebras that satisfy (C').
\end{as}

The following Lemma shows that this could be slightly weakened.

\begin{lemma}The algebra $R$ satisfies (C') if and only if the algebra $\Gamma$ satisfies (C').
\end{lemma}
\begin{proof} The proof is analogous to \ref{periodic}.
\end{proof}
We also have
\begin{lemma}\label{periodic1}Suppose $\Gamma \not =S$. Let $E$ be a non-projective $\Gamma$-module and $N$ a simple $R$-module. Then $\underline{\Hom}(S,E) \not =0$ and $\underline{\Hom}(N, E^R) \not =0$.
\end{lemma}
\begin{proof} We consider the map which maps $S$ to $\soc E$.
This map does not factor through $\Gamma$. Therefore $\underline{\Hom}(S,E) \not =0$. We have $\Hom(N, E^R) \cong \Hom(N_{\Gamma}, E) \cong \Hom(S, E)$ by \ref{frob}. As the restriction to $\Gamma$ and lifting to $R$ preserves projectivity, we have $\underline{\Hom}(N, E^R) =\underline{\Hom}(S,E) \not =0$.
\end{proof}

We can now investigate the relation between the Auslander-Reiten components of $T_s(\Gamma)$ and the
Auslander-Reiten components of $T_s(R)$.  We assume for the
next two Theorems that \ref{as2} is satisfied. We
denote by $Obj(\theta)$ the set of all indecomposable modules in
an Auslander-Reiten component $\theta$.

\begin{theo}\label{smash euc}
(1) $T_s(\Gamma)$ has a component of tree class $\tilde D_n$ if
and only if $T_s(R)$ has a component of tree class $\tilde D_n$.

(2) Suppose $T_s(R)$ has a component of tree class
$\tilde A_{1,2}$, then $T_s(\Gamma)$ has also a component of
tree class $\tilde A_{1,2}$.
\end{theo}
\begin{proof}
(a) Let $\theta\cong \Z [\tilde D_n]$ or $\theta\cong \Z [\tilde
A_{1,2}]$ be a component of $T_s(R)$. We denote the tree class of $\theta$ by $\mathcal{T}$.
 Then by \ref{simple}
$\theta$ contains a simple module $M$ and a projective module $P$
is attached to $\theta$.  Let $\Delta$ be the component in
$T_s(\Gamma)$ containing the simple module $S=M_{\Gamma}$. Then
by \ref{aus restrict} we have $Obj(\theta_{\Gamma} )\subset
\cup_{g\in G} Obj(\Delta_g)$ using induction on the distance of a module in $\theta$ to $M$.

\smallskip

 As $S$ is contained in $\Delta$ and $G$ acts
trivially on $S$, we have $\Delta_g=\Delta$ for all $g\in G$ and therefore
$Obj(\theta_{\Gamma} )\subset Obj( \Delta)$. As $P$ is attached to $\theta$,
 $\Gamma$ is attached to $\Delta$. Then $\Omega $ induces a fix point free
automorphism on the tree class $T$ of $\Delta$.

\smallskip

 By \ref{euc1} and \ref{euc2},
 $\mod R$ has a periodic module that does not lie in
$\theta$ and therefore by \ref{periodic}, we have that $\mod
\Gamma$ also contains a periodic module that does not lie in
$\Delta$. Using \ref{periodic1} we can define a subbaditive,
non-zero function on the component $\Delta$ as in \cite[3.2]{ES}.
The tree class of $\Delta$ is therefore in the HPR-list (this is a
list of trees given in \cite[p.286]{HPR}).

\smallskip

 As we have a fix point free automorphism operating on $T$, this gives
$T=A_{\infty}^{\infty}$, $T=\tilde D_m$ for $m $ odd or $T=\tilde
A_{1,2}$.

\smallskip

Note that $S$ and $\Gamma/S$ do not lie in the same $\tau $-orbit, because by \ref{length} we would have $2=l(\Gamma)$.
But then $S$ would be periodic which is a contradiction to the fact that $M$ is not periodic.

\smallskip

Twisting with $g \in G$ also induces an automorphism of
finite order on $\Delta$, that fixes $S$ and $\Gamma/S$.  If $T=\tilde D_m$, then $S$ and $\Gamma/S$ have only one predecessor and their predecessors do not lie in the same $\tau$-orbit by \ref{d}. Therefore $g$
acts as the identity and there are no modules contained in $\Delta$
that are twists of each other.

This means that we can embed $\Delta$ into $\theta$ using induction on the distance of a module in $\Delta$ to $S$. We give a sketch of how to construct this injection: We map $S$ to $M$. Let $W \in Obj(\Delta) $ and suppose that there is an arrow from $W$ to $S$  in $\Delta$. Then  by \ref{restrict aus} there exists an $R$-module $J$ such that $J$ is a summand of the middle term of $A(M)$ and $W|  J_{\Gamma}$. We map $W$ to $J$. This gives an injection, as for two indecomposable $\Gamma$ modules $W^1$ and $W^2$ with $W^1| Z_{\Gamma}$ and $W^2|Z_{\Gamma} $ for an indecomposable $R$-module $Z$, we have $W^1 \cong (W^2)_g$ for some $g \in G$. So if $W^1$ and $W^2$ lie in $\Delta$ we have $W^1 \cong W^2$.

\smallskip

As this embedding respects $\tau$, it induces an embedding $T \subset \mathcal{T}$.  Therefore $T= \mathcal{T}$ . This proves the first
direction of $(1)$ and part $(2)$.

\medskip

(b) Suppose now that $\Delta\cong  \Z [\tilde D_n]$, then $S\in
\Delta$ and $\Gamma$ is attached to $\Delta$. Let $\theta$ be a component of $T_s(R)$ that contains a simple
module $M$. As in (a) we have
$Obj(\theta_{\Gamma} )\subset \cup_{g\in G} Obj(\Delta_g)$.

\smallskip

By \ref{euc2}, $\mod \Gamma$ contains a periodic module $E$ that is
not in $\Delta$ and by \ref{periodic} all direct summands of $E^R$ are periodic and do not lie in $\theta$.
As in the first part of the proof, this shows that the tree class $T$ of
$\theta $ is in the HPR-list.

\smallskip

 As $\bar \alpha(S)=\bar \alpha(\Gamma / S)=1$ by \ref{d} and do not have the same predecessor, $g$ acts as the
identity on $\Delta$. That means $Obj(\theta_{\Gamma})\subset Obj( \Delta)$.


As in (a) we can embed $\Delta$ into $\theta$. As $\tau$ and
the restriction to $\Gamma$ commute by \ref{restrict aus}, this gives an
embedding of $\tilde D_n$ into $T$. Therefore $T=\tilde D_n$.
\end{proof}
Note that the embedding of $\Delta$ into $\theta$ does not respect labels of arrows.
The next example shows that part $(2)$ of the previous theorem
does not hold in the converse direction.
\begin{example}
Let $k$ be a field of characteristic 2, $\Gamma=kV_4$  and $R
=kA_4 \cong kV_4 \rtimes k C_3$.
 Then $\Gamma$ has a component with tree class $\tilde A_{1,2}$ and $R$ has a component with reduced graph $\tilde A_5$
  which corresponds to a tree class $A_{\infty}^{\infty}$. By \ref{euc1} $T_s(R)$ has no component of Euclidean tree class.
\end{example}
We can also show the following

\begin{theo}\label{12}
Suppose $T_s(\Gamma)$ has a component $\Delta$ of tree class $\tilde
A_{1,2}$, then $T_s(R)$ has also a component $\theta$ of tree
class $T=\tilde A_{1,2}$ or $T=A_{\infty}^{\infty}$. In
the second case $\theta \cong \Z[\tilde A_n]$, where $n+1/2$ divides
the order of $G$.
\end{theo}
\begin{proof}
Let $M$ be a simple $R$-module and let $\theta$ be the component
that contains $M$. As in the proof of the previous Theorem the
tree class $T$ of $\theta$ is from the HPR-list and $\theta$ is
not a periodic component. Let $Q$ be the middle term of $A(M)$.
Then $Q_{\Gamma}=N\oplus N$ for some indecomposable
$\Gamma$-module $N$.

\smallskip

 Suppose first that $Q\cong L\oplus L_g$ for $L$
an indecomposable $R$-module and $g\in G$ such that $L_{\Gamma}
\cong N$ and $L_g \not \cong L$. Then $\bar \alpha(M)=2$ and $g$ acts as a graph
automorphism on $T $ of finite order and is not the identity. Also this automorphism commutes with $\tau$. Furthermore the middle term of $A(\tau^{-1}(L))$ is $M \oplus M_{g^{-1}}$.
As $M \not \cong M_{g^{-1}}$ we have $\theta \cong\Z[\tilde A_n]$, where $n=2|g|-1$.

\smallskip

Suppose now that $Q$ is indecomposable. Then either $M\oplus M$ or
$M\oplus M_g$ is the non-projective summand in the middle term of
$A(\tau^{-1}(Q))$. In the first case, we have $\bar \alpha(M)=\bar \alpha(Q)=1$. Therefore $T=Q\stackrel{(1,2)} \to  M$. This contradicts with the HPR-list. In the second case we have $\bar \alpha(Q)=2$ and $\bar \alpha(M)= \bar \alpha(M_g)=1$. As all three are in different $\tau$-orbits, we have $T= \xymatrix{ M \ar[r]& Q &M_g \ar[l]}$ which is also a contradition, as $\Delta$ is not of finite type.

\smallskip

Suppose now that $Q\cong L\oplus L$. Then either $M\oplus M$ or $M\oplus M_g$ is the
non-projective summand in the middle term of $A(\tau^{-1}(L))$.
The first case gives $\theta \cong \Z[A_{1,2}]$. In the second case we have $T= \tilde B_2$, which we can exclude using \ref{euc1}.
\end{proof}

 \end{document}